\documentclass[letterpaper,12pt]{amsart}
\usepackage{hyperref}
\usepackage{latexsym}
\usepackage{amssymb}
\usepackage[cp850]{inputenc}
\usepackage{epsfig}
\usepackage{amsthm}
\usepackage{amscd}
\usepackage{amsmath}
\usepackage{amsfonts}
\usepackage{graphics}
\usepackage[all]{xy}
\usepackage{dsfont}
\usepackage{epigraph}

\newtheorem{theorem}{Theorem}
\newtheorem{lemma}[theorem]{Lemma}

\newtheorem{prop}[theorem]{Proposition}

\newtheorem*{theorem*}{Theorem}
\newtheorem*{corollary*}{Corollary}
\theoremstyle{definition}

\newtheorem*{remark*}{Remark}
\newtheorem{question}[theorem]{Question}

\newtheorem*{definition*}{Definition}
\newtheorem*{example*}{Example}

\numberwithin{theorem}{section}

\newcommand{\BR}{\mathbb R} 
 
 \newcommand{\BZ}{\mathbb Z}

\newcommand{\CI}{\mathcal I}

\newcommand{\wt}{\widetilde}
\newcommand{\wh}{\widehat}
\newcommand{\nid}{\noindent}

\newcommand{\tup}{\textup}

\DeclareMathOperator{\GL}{GL}

\newcommand{\comment}[1]{}

\usepackage{etoolbox}

\makeatletter
\patchcmd{\epigraph}{\@epitext{#1}}{\itshape\@epitext{#1}}{}{}
\makeatother

\begin{document}
\title  [Separating Subgroups]  {Separating subgroups of mapping class groups in homological representations}
\author   {Asaf Hadari}
\date{\today}
\begin{abstract} Let $\Gamma$ be either the mapping class group of a closed surface of genus $\geq 2$, or the automorphism group of a free group of rank $\geq 3$. Given any homological representation $\rho$ of $\Gamma$ corresponding to a finite cover, and any term $\CI_k$ of the Johnson filtration, we show that $\rho(\CI_k)$ has finite index in $\rho(\CI)$, the Torelli subgroup of $\Gamma$. Since $[\CI: \CI_k] = \infty$ for $k > 1$, this implies for instance that no such representation is faithful.   \end{abstract}
\maketitle

\vspace {10mm}

\section{Introduction} \label{intro}

Let $X = X^n_{g,b}$ be an oriented surface of genus $g$ with $b$ boundary components and $n$ punctures.  The mapping class group of $X$, or  $\tup{Mod}(X) = \tup{Mod}^n_{g,b}$ is the group of orientation preserving diffeomorphisms $X \to X$ that fix the boundary and punctures pointwise, up to isotopies that fix the boundary point wise. 
 
Mapping class groups have a large collection of finite dimensional representations called \emph{homological representations}. For every finite characteristic cover $f: Y \to X$ we can associate a representation $\rho = \rho_f: \tup{Mod}^{n+1}_{g,b} \to \tup{GL}(H_1(Y, \BZ))$ in the following way.

Pick a point $\star \in X$. Suppose $f: Y \to X$ corresponds to  characteristic subgroup $K \leq \pi_1(X, \star)$. Let $\Gamma \cong \tup{Mod}^{n+1}_{g,b}$ be the mapping class group of the surface $X'$ which is obtained from $X$ by adding a puncture at $\star$. We can think of $\Gamma$ as the group of orientation preserving diffeomorphisms $X \to X$ that fix the boundary, punctures, and $\star$ pointwise up to isotopies that fix the boundary and $\star$ point wise. There is a natural map $\Gamma \to \tup{Aut}(\pi_1(X, \star))$. Since, $K$ is a characteristic subgroup, restriction gives a map $\Gamma \to \tup{Aut(K)}$. This induces a map $\Gamma \to \tup{Aut}(H_1(K, \BZ)) \cong \GL(H_1(Y, \BZ))$. 

Our goal in this paper is to address two basic questions about homological representations. These questions fit into a larger meta-question: 

\begin{question} \label{urq} Which properties of the mapping class group can we discern in its homological representations? 
\end{question}

\nid The main question question we address in this paper is the following: 

\begin{question} \label{q1}
Let $H \leq G \leq \Gamma$ be subgroups of $\Gamma$ such that $[G:H] = \infty$. Is there a homological representation $\rho_f$ such that $[\rho_f(G):\rho_f(H)] = \infty$? In other words, can the homological representation theory of mapping class groups discern whether or not one subgroup has infinite index in another?
\end{question}

\nid When $H$ is the trivial group, the answer to question \ref{q1} is yes. For example in \cite{eigoffuc} we showed that if the surface $X$ has boundary components then any element of $\Gamma$ with positive topological entropy has infinite order in some homological representation (and in fact, has eigenvalues off of the unit circle). Yi Liu proved a similar theorem in \cite{YiLiu} that holds for closed surfaces as well. In \cite{nonsolv}, we proved an even stronger result - if $H$ is the trivial group and $G$ is a non-amenable subgroup of $\Gamma$, then there is some homological representation $\rho_f$ such that $\rho_f(G)$ is non-amenable (and in particular infinite). 

This paper is concerned with the opposite direction - what happens when $H$ is a large subgroup of the mapping class group? I turns out that in this case a very different phenomenon occurs in which the answer to Question \ref{q1} can be negative. Namely, we show the following: 

\begin{theorem} \label{theorem1} Suppose that $X$ is a closed surface of genus $\geq 2$. Then there exist subgroups $H \leq G \leq \Gamma$ such that $[G:H] = \infty$ and $[\rho(G):\rho(H)] < \infty$ for every homological representation $\rho$.  
\end{theorem} 

The subgroups we use in our proof of Theorem \ref{theorem1} are well known subgroups. Let $\CI = \CI_1$ be the Torelli subgroup of $\Gamma$, and $\CI_k$ be the $k$-th term in the Johnson filtration (these groups are defined below in \ref{Jfil}). We prove the following:

\begin{theorem} \label{theorem2} Let $X$ be a closed surface of genus $\geq 2$.For every $k > 0$, and every homological representation $\rho$ of $\Gamma$, $[\rho(\CI):\rho(\CI_k)] < \infty$. 
\end{theorem}

\nid This is enough to prove theorem \ref{theorem1} since $[\CI: \CI_k] = \infty$ for every $i \geq 2$. 

\smallskip

We also prove an analog of Theorem \ref{theorem2} for $\tup{Aut}(F_n)$ with $n \geq 3$. Given any characteristic subgroup $K < F_n$, we get a representation $\rho: \tup{Aut}(F_n) \to \tup{GL}(H_1(K, \BZ))$. The group $\tup{Aut}(F_n)$ also has subgroups $\CI_k$ (which we define in \ref{Jfil}). We prove the following theorem: 

\begin{theorem}\label{theorem3} Let $n \geq 3$. For every $k$ and every homological representation $\rho$ of $\tup{Aut}(F_n)$, $[\rho(\CI):\rho(\CI_k)] < \infty$. 
\end{theorem}
 
As part of our proof of Theorem \ref{theorem2}, we answer an even more basic question whose answer was surprisingly not in the literature - namely: is any homological representation faithful? 

\begin{theorem} \label{theorem4} If $X$ is a closed surface of genus $\geq 2$ then no homological representation is faithful. 

\end{theorem}

Theorem \ref{theorem4} follows from Theorem \ref{theorem2} since $[\CI:\CI_2] = \infty$. It's also proved directly in Lemma \ref{mainlemma}. 

Our final observation is that while our theorems show that the image of one group has finite index in the other, this index need not be $1$. 

\begin{theorem}\label{theorem5}
	Let $\CI, \CI_2, \CI_3, \ldots$ be the Johnson filtration of either $\tup{Mod}(X)$ for $X$ a closed surface of genus $\geq 2$ or $\tup{Aut}(F_n)$ for $n \geq 2$. Then there exists a homological representation $\rho$ with the following property:  for any $k > 0$ there exists a number $N \geq k$ such that if $j > N$ then $[\rho(\CI_k): \rho(\CI_j)] > 1$. 
\end{theorem}

\subsection{the image of homological representations} \label{Image}
Theorem \ref{theorem2} fits a conjectural description of the image of homological representations. Suppose $f: Y \to X$ is a finite characteristic cover with deck group $D$. Given a mapping class $\phi \in \Gamma$, any lift $\wt{\phi}$ of $\phi$ to $Y$ normalizes the deck group $D$. In particular, $\rho_f(\phi)$ normalizes $D_*$, the image of $D$ in $\tup{Sp}(H_1(Y, \BZ))$. It is natural to ask the following question: 

\begin{question} \label{q2} Is $\rho_f(\Gamma)$ a finite index subgroup of the normalizer of $D_*$ in $\tup{Sp}(H_1(Y, \BZ))$? Does a similar phenomenon hold for homological representations of $\tup{Aut}(F_n)$?
\end{question}

McMullen addressed this question in \cite{McM}. He showed that when the genus of $X$ is zero, the answer to question \ref{q2} can be negative. However, in every single known case in genus $\geq 2$, the answer to this question is positive. For example, Gr\"unewald, Larsen, Lubotzky, and Malestein showed this for the class of redundant covers of closed surfaces in \cite{GLLM} and Looijenga showed it for the class of abelian covers of closed surfaces in \cite{Looij}. The corresponding question for $\tup{Aut}(F_n)$ representations also has a positive answer in every single known case (for example, Gr\"unewald and Lubotzky proved this for the class of redundant covers in in \cite{GL}). 

Our Theorems \ref{theorem2} and \ref{theorem3} can be viewed as evidence of a positive answer to Question \ref{q2} in genus $\geq 2$ and for $F_n$ with $n \geq 3$. The normalizers that appear in the question are lattices in high rank semi-simple Lie groups. Since $\tup{Mod}(X) / \CI \cong \tup{Sp}(2g,\BZ)$, we get that a positive answer to the mapping class group portion of Question \ref{q2} implies that the image of $\CI$ is also a lattice in a high rank semi-simple Lie group. By the Margulis normal subgroup theorem, every normal subgroup of such a lattice is either finite or has finite index. In particular, all of its quotients are either finite or semi-simple. Since $\CI_k \lhd \CI$ for every $k$ and $\CI/\CI_k$ is a nilpotent group, we must have that the image of $\CI_k$ has finite index in the image of $\CI$. Thus, a positive answer to Question \ref{q2} would imply our Theorem \ref{theorem2}.

\nid \textbf{Acknowledgments} We would like to thank Dan Margalit for helpful comments about the paper. 

\subsection{Sketch of Proof of Theorem \ref{theorem2}}\label{proofsketch}

The group $\Gamma$ acts on $\CI/ \CI_2$ by conjugation. Since $\CI$ acts trivially on this quotient, $\CI/\CI_2$ becomes a $\tup{Sp}(2g,\BZ)$-module. Denote $H = H_1(X, \BZ)$. By work of Dennis Johnson, it is known that $\CI/\CI_2 \cong \Lambda^3 H$ as a $\tup{Sp}(2g,\BZ)$-representation (see \cite{Johnson1}). 

This representation decomposes as a sum of irreducibles as $\Lambda^3 H \cong H \oplus \Lambda^3H/H$. The first factor is the image of the point pushing subgroup and the second factor is the image of those elements of $\CI$ that are not in the kernel of the map $\tup{Mod}(X') \to \tup{Mod}(X)$ given by forgetting the puncture $\star$. 

In Section \ref{pushmaps} we discuss point pushing maps and a related notion called curve pushing maps. We also give an explicit description of the image of certain push maps under homological representations. In Section \ref{proof}, we use this description to construct for each homological representation $\rho$ elements $\phi, \psi \in \Gamma$ such that: 

\begin{enumerate}
\item $\phi, \psi \in \CI \setminus \CI_2$. 
\item $\phi, \psi \in \ker(\rho)$ (which immediately proves Theorem \ref{theorem4}) 
\item $\phi$ is a point pushing map that projects to the first irreducible factor of $\CI/ \CI_2$. 
\item $\psi$ is a curve pushing map that projects to the second irreducible factor of $\CI / \CI_2$. 
\end{enumerate}

\nid Since $\phi, \psi \in \ker(\rho)$, their conjugacy classes are as well. This means that the $\tup{Sp}(2g, \BZ)$ orbits in $\CI /\CI_2$ are in the kernel of the map $\CI \to \rho(\CI)/ \rho(\CI_2)$. Irreducibility now gives the result. 

For $k \geq 2$, we have that $\CI / \CI_k$ is a nilpotent group. In Lemma \ref{nilp}, we show that if $N$ is a finitely generated nilpotent group and $K \lhd N$ projects to a finite index subgroup of $N/[N,N]$ then $K$ has finite index in $N$. Applying this to $\ker(\rho)$ now gives the result. 

In Section \ref{autcase} we carry out a similar proof for the $\tup{Aut}(F_n)$ case. This case is much simpler. The representation $\Lambda^3H$ is an irreducible $\tup{SL}(n,\BZ)$-representation, so we only need to construct one map $\phi \in \CI \setminus \CI_2 \cap \ker(\rho)$. Furthermore, this map ends up being easier to construct - it's a Nielsen transformation that is simple to describe. 
\section{The Torelli group and the Johnson filtration} \label{Jfil}

Let $\Gamma$ be either $\tup{Aut}(F_n)$, or $\tup{Mod}(X')$ where $X$ is a closed surface of genus $\geq 2$ and $X'$ is obtained from $X$ by adding the puncture $\star$. Let $\pi$ be either $F_n$ or $\pi_1(X, \star)$. 

There is a natural map $\Gamma \to \tup{Aut}(\pi)$. Whenever $L \lhd \pi$ is a characteristic subgroup, we get a map $\Gamma \to \tup{Aut}(\pi/L)$. The sequence of groups $L_1 = [\pi, \pi]$, $L_{i+1} = [\pi, L_i]$ is called the \emph{lower central series of $\pi$}. All the groups $L_i$ are characteristic subgroups of $\pi$. 

The kernel of the map $\Gamma \to \tup{Aut}(\pi/L_k)$ is called the \emph{k-th term of the Johnson filtration}. We denote this kernel $\CI_k$. When $k =1$ the group is called the \emph{Torelli subgroup of $\Gamma$} and we denote it $\CI$.

We will require three standard facts about the Johnson filtration: 

\begin{enumerate}
\item $\CI / \CI_k$ is a finitely generated nilpotent group.  

\item The group $\CI/\CI_2$ is abelian and the map $\CI/ [\CI, \CI] \to \CI/\CI_2$ has finite kernel. (\cite{Johnson2}). 
\item The group $\Gamma$ acts on $\CI/\CI_2$ by conjugation. As a $\Gamma/\CI$-module, $\CI/CI_2 \cong \Lambda^3 H$, where $H = \pi/[\pi, \pi]$. 

\end{enumerate}

\section{The $\tup{Aut}(F_n)$ case} \label{autcase}

Let $n \geq 3$, $K \lhd F_n$ be a characteristic subgroup, and $\rho_K: \tup{Aut}(F_n) \to \tup{GL}(H_1(K, \BZ))$ be the corresponding homological representation. 

\begin{prop} \label{prop1}  In the notation above $[\rho_K(\CI): \rho_K(\CI_2)] < \infty$. 
\end{prop}

\begin{proof}
Let $m = [F_n : K]$, and $F_n = \langle a_1, a_2, \ldots, a_n \rangle$. For every fixed $1 \leq i \neq j \leq n$ the endomorphism $F_n \to F_n$ that sends $a_i \to a_i a_j$ and $a_k \to a_k$ for $k \neq i$ is called a \emph{Nielsen transformation}, and is well known to be an automorphism. Let $\phi: F_n \to F_n$ be given by $\phi(a_1) = a_1 [a_2^m, a_3^m]$, and $\phi(a_k) = a_k$ for $k > 1$. The endomorphism $\phi$ can be written as a product of Nielsen transformations, and is thus an automorphism. \\

\nid \underline{Claim 1:} $\phi \in \tup{Ker}(\rho_K)$. To see this claim, let $X = \bigvee \limits_1^n S^1$ be a join of $n$ circles at a point which we call $p$. Then $\pi_1(X,p) \cong F_n$, and the automorphism $\phi$ is induced by a homotopy equivalence $\varphi: X \to X$ that fixes $p$. Let $X_0 \to X$ be the finite sheeted cover corresponding to the subgroup $K < F_n$. Since $K$ is characteristic, we can lift $\varphi$ to a map $\varphi_0: X_0 \to X_0$ fixing some lift $p_0$ of $p$.  \\

\nid The spaces $X$ and $X_0$ are graphs, and are thus  also simplicial complexes. Let $C_*$ be the chain complex of simplicial chains in $X_0$. Denote by $\varphi_0^\#: C_1 \to C_1$ the map induced by $\varphi_0$. Note that $a_2^m, a_3^m \in K$. Thus, $[a_2^m, a_3^m] \in [K,K]$. Given an edge $r$ in $X_0$ that is a lift of an edge corresponding to $a_i$ for $i >1$, we have that $\varphi_0(r) = r$. Given an edge $r$ that is a lift of $a_1$, we have that $\varphi_0(r)$ is an edge path consisting of $r$ followed by a cycle whose class in $H_1(X_0)$ is trivial.  Thus $\varphi_0^\#$ is the identity map. \\

\nid Since $H_1(K)$ can be identified with the subspace of $1$-cycles in $C_1$, and $\rho_K(\phi)$ is given by restricting $\varphi_0^\#$ to this subspace, we get that $\rho_k(\phi)$ is the identity map. 

\nid \underline{Claim 2:} $\phi \in \CI \setminus \CI_2$. Note that , $a_1^{-1}(\phi(a_1)) = [a_2^m, a_3^m]$ and $a_i^{-1} \phi(a_i)$ for $i \geq 2$. Since $[a_2^m, a_3^m] \in [F_n, F_n]$ we get that $\phi \in \CI$. Since $[a_2^m, a_3^m] \notin [F_n, [F_n, F_n]]$ we get that $\phi \notin \CI_2$. 

We are now ready to complete the proof. The quotient $\CI/ \CI_2$, as a $\tup{Aut}(F_n)$ module is $\Lambda^3 H$. Since $\CI$ acts trivially on this module, $\CI / \CI_2$ is a $\tup{Aut}(F_n) / \CI \cong SL(n,\BZ)$-module. 

Let $[\phi]$ be the image of the automorphism $\phi$ constructed above in $\CI / \CI_2$. By claim 1, $[\phi]$ is in the kernel of the map $\CI / \CI_2 \to \rho_K(\CI) / \rho_K(\CI_2)$.  Thus, the entire $\tup{SL}_n(\BZ)$ orbit of $[\phi]$ is in this kernel. By claim 2, $[\phi] \neq 0$.

Since $\Lambda^3 H$ is an irreducible $\tup{SL}(n,\BZ)$ representation, we have that $\tup{SL}(n,\BZ) [\phi]$ generates a finite index subgroup of $\Lambda^3 H$. Thus $\rho_K(\CI) / \rho_K(\CI_2)$ is finite, as desired.

\end{proof}

\nid Proposition \ref{prop1} is a subcase of Theorem \ref{theorem3}, namely - the case where $k=2$. Since $\CI/ \CI_k$ is nilpotent for every $k$, Theorem \ref{theorem3} follow directly from Proposition \ref{prop1} and from the following lemma: 

\begin{lemma} \label{nilp}
Let $N$ be a finitely generated nilpotent group. Let $K \lhd N$ be a subgroup such that the image of $K$ in $N / [N,N]$ has finite index in $N / [N,N]$. Then $[N:K] < \infty$.

\end{lemma}

\begin{proof}
Denote by $N_i$ the $i$-th term in the lower central series of $N$. Let $a_1, \ldots, a_m \in N$ be a set of elements that project to a free generating set of $N/[N,N]$.  For elements $g_1, \ldots, g_k \in N$, denote by $[g_1, \ldots, g_k]$  the repeated commutator  $[g_1, [g_2, \ldots, g_k] \ldots ]$. The following two facts are standard: 

\begin{enumerate}
\item The group $N_1/N_{i+1}$ is an abelian group generated by the images of all elements of the form $[x_1, \ldots, x_i]$ where $x_1, \ldots, x_i \in \{a_1, \ldots, a_n \}$. 
\item The map $N^i \to N_i /N_{i+1}$ given by $(x_1, \ldots, x_i) \to [x_1, \ldots, x_i]$ is a homomorphism in each coordinate. 
\end{enumerate}

\nid We now proceed by induction on the length of the lower central series of $N$. The lemma is obviously true for $N$ abelian. Suppose it is true for $N$ with a lower central series of length $i$. The multi linearity of the repeated commutator, together with our inductive hypothesis give that $[N_i: N_i \cap K] < \infty$. 

Given any $x \in N$, the inductive hypothesis gives that there exists a number $l$ such that the image of $x^l$ in $N / N_{i-1}$ is contained in the image of $K$ in this group. This means that there is a $y$ such that $y^l \in K$ and $xy^{-1} \in N_i$. Write $x = yh$. Since $[N_i: N_i \cap K] < \infty$, there is a $m$ such that $h^m \in K$. Since $N_i$ is central in $N$, we have that $(yh)^{lm} = (y^{l})^m (h^m)^l \in K$. A finitely generated nilpotent group of bounded exponent is finite, which concludes the proof.  

\end{proof}

\section{Pushing maps} \label{pushmaps}
\subsection{Point pushing and curve pushing}

\nid We begin by recalling a standard theorem from differential topology: 

\begin{theorem} (\textup{The isotopy extension theorem})
Let $M$ be a compact manifold (possibly with boundary) and $N$ a boundary less sub manifold. Let $H: N \times [0,1] \to M$ be a smooth homotopy such that $H_t(x) = H(x,t): N \to M$ is an embedding for each $t$ and $H_0$ is the inclusion of $N$ into $M$. Then $H$ can be extended to a smooth isotopy $\wt{H}: M \times [0,1] \to M$ where $\wt{H}_t(x) = \wt{H}(x,t): M \to M$ is a diffeomorphism for each $t$, and $\wt{H}_0$ is the identity map. 
\end{theorem}

Extensions of homotopies $H$ that have the added property that $H_1$ is the identity map on $N$ allow us to construct several interesting families of mapping classes, which we will use to mimic the proof of Theorem \ref{theorem3} for mapping class groups. 

The first such family is very well known. Suppose $M$ is a surface, and $p \in M$ is a point. Let $N = \{ p\}$. A homotopy H satisfying the condition $H_0 = H_1 = \tup{Id}$ is just a closed curve $\gamma$ that is based at $p$. Let $\wt{H}$ be the extended isotopy. Let $M_0 = M \setminus \{ p\}$.  The diffeomorphism $\wt{H}_1$ can be restricted to $M_0$. This restriction is known as the \emph{point pushing map about the curve $\gamma$}.  By construction $\wt{H}_1$ is isotopic to $\wt{H}_0 = \tup{Id}$ as maps from $M \to M$. However, if $\gamma$ is not null-homotopic then their restrictions are not isotopic as maps $M_0 \to M_0$. 

Now, suppose $N \subset M$ consists of a finite set of points $N = \{p_1, \ldots, p_r\}$ and once again assume that $H_0 = H_1$. This means that each point $p_i$ traces a closed curve $\gamma_i$. We call the restriction of the diffeomorphism $\wt{H}_1$ to $M_0 = M \setminus M$ a \emph{multi-point pushing map about the curves $\gamma_1, \ldots, \gamma_r$}. 

Let $X$ be a surface and $\delta$ a non-peripheral, non-separating simple closed curve. Let $X_0$ be the surface obtained from $X$ by cutting along $\delta$. The surface $X_0$ has two more boundary components (which we call $\delta_1, \delta_2$) than $X$, and its genus is one less than the genus of $X$. Let $\overline{X}$ be the surface obtained from $X_0$ by gluing a disk to $\delta_1$, and adding a marked point $p$ in this disk. 

Choose a closed curve $\gamma$ in $\overline{X}$ that is based at $p$. The curve $\gamma$ gives a homotopy of embeddings of the point $p$ into $\overline{X}$. Extend this homotopy to a isotopy $\wt{H}: \overline{X} \times [0,1] \to \overline{X}$. If $B$ is a sufficiently small ball centered at $p$, we can modify $\overline{H}$ so that $\overline{H}_1$ is the identity map when restricted to $B$. 

Let $\overline{f} = \wh{H}_1$. The surface $X$ can be obtained from $\overline{X}$ by removing the interior of $B$, and gluing $\partial B$ to $\delta_2$. Since $\overline{f}$ fixes $\partial B$ and $\delta_2$ point wise, it defines a map $f: X \to X$, which we call a \emph{curve pushing map}. We say that $f$ \emph{pushes the curve $\delta$ along the curve $\gamma$}. 

Note that $\gamma$ is not a closed curve in $X$, but we can think of it as a closed curve in $(X, \delta)$. Throughout our discussion, when we refer to $\gamma$ as a closed curve we mean it in this sense. 

We can make a similar definition when the curve $\delta$ is separating. Suppose it separates $X$ into two components, $X_0, X_1$. We choose one of them, say $Z_0$ and glue in a disk to $X_0$ along $\delta$ equipped with a marked point $p$. We pick a closed curve $\gamma$ in $X_0$ that is based at $p$ and proceed as before. 

By replacing the curve $\delta$ with a multi-curve (a finite collection of mutually disjoint simple closed curves) we can define a \emph{mutli-curve pushing map}.

\subsection{The image of push maps under homological representations}
\subsubsection{The action of point and multi-point pushing maps on homology}

Let $X$ be a surface with at least one puncture. Let $\phi \in \tup{Mod}(X)$ be a point pushing map about the closed curve $\gamma$ which is based at the puncture $p$ and whose homology class (relative to the puncture $p$) we denote by $c$. Let $d$ be the homology class of a small positively oriented loop about the puncture $p$. If the surface $X$ has only one puncture, then every point pushing map acts trivially on $H_1(X, \BZ)$. This no longer holds when $X$ can have multiple punctures.

Suppose first that $\gamma$ is a simple closed curve. Let $\alpha$ be a closed curve in $X$ with homology class $a$. The curve $\phi(\alpha)$ is obtained from $\alpha$ in the following way. For every intersection of $\gamma$ and $\alpha$, two strands that run alongside  $\gamma$ in opposite orientations are attached to $\alpha$ using the surgery depicted in the figure $1$. If we use $[\cdot]$ to denote the homology class of a curve, and $\wh{i}(\cdot, \cdot)$ to denote the oriented intersection pairing we get that: 

$$\phi(a) = a + \wh{i}(a,c)d $$

\begin{figure}[h]
	\begin{center}
		\includegraphics[width=1\textwidth]{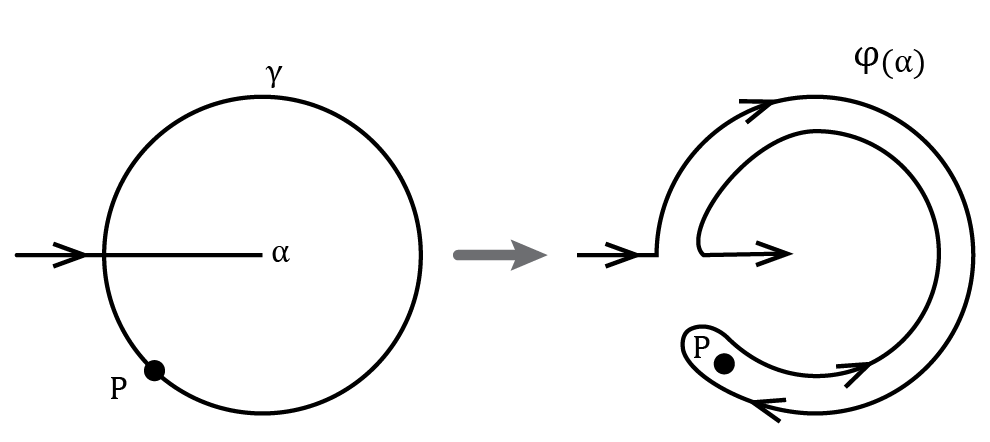}
	\end{center}
	\caption{}
\end{figure}

Suppose now that the curve $\gamma$ has self intersections. Perform the surgery described above by attaching two copies that run along $\gamma$ in opposite orientations to each intersection point. Label the intersection points cyclically by $x_1, \ldots, x_s$. At the intersection point $x_i$ add $2i$ strands ($i$ in each orientation) parallel to the portion of $\gamma$ that exits the intersection the second time it pases through it. To find the image $\phi(c)$, at each self intersection of $\gamma$ perform the surgery described in Figure $2$. At each such intersection there are an equal number of positively oriented and negatively oriented strands parallel to $\gamma$ in each direction. Thus, even if $\gamma$ has self intersection, the same formula holds: 

$$\phi(a) = a+ \wh{i}(a,c )d $$

\begin{figure}[h]
	\begin{center}
		\includegraphics[width=1\textwidth]{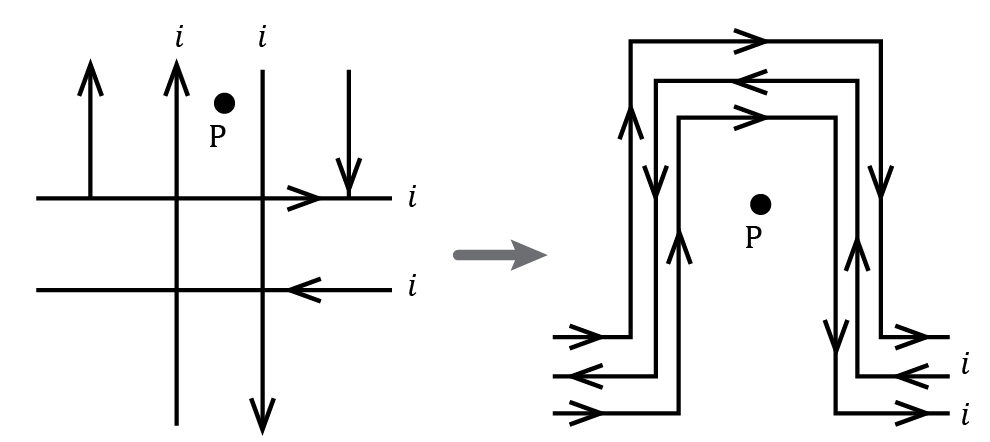}
	\end{center}
	\caption{}
\end{figure}

Now suppose that $\phi$ is a multi-point push map, in which the punctures $p_1, \ldots, p_r$ are pushed about the curves $\gamma_1, \ldots, \gamma_r$ (whose homology classes we denote by $c_1, \ldots, c_r$). Let $d_1, \ldots, d_r$ be small, positively oriented loops about the punctures $p_1, \ldots, p_r$. The same argument that we used for the point pushing map gives a similar description (this is illustrated in Figure $3$). 

\begin{lemma} \label{multipoint}
	In the notation above, for every curve $\alpha$, we have that: 
	$$\phi(a) = a + \sum_{i=1}^r\wh{i}(a,c_i )d_i $$
\end{lemma} 

\begin{figure}[h]
	\begin{center}
		\includegraphics[width=1\textwidth]{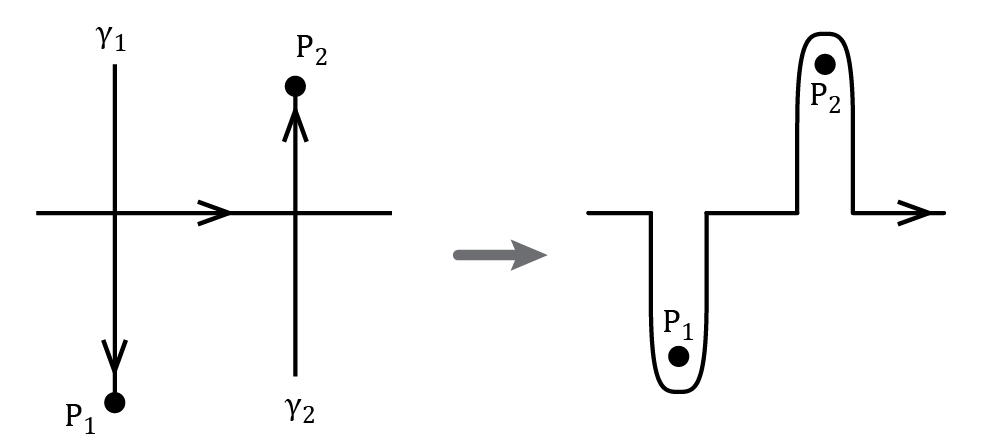}
	\end{center}
	\caption{}
\end{figure}

\subsubsection{The action of curve and multi-curve pushing maps on homology}
Suppose that $\phi$ is the map given by pushing the curve $\delta$ (whose homology class we denote by $d$) along the curve $\gamma$. For any curve $\alpha$ that does not intersect $\delta$, the same picture as the point pushing case holds here (this is illustrated in Figure $4$). We get that once again, $\phi(a) =a + \wh{i}(a,c ) d $. 

\begin{figure}[h]
	\begin{center}
		\includegraphics[width=1\textwidth]{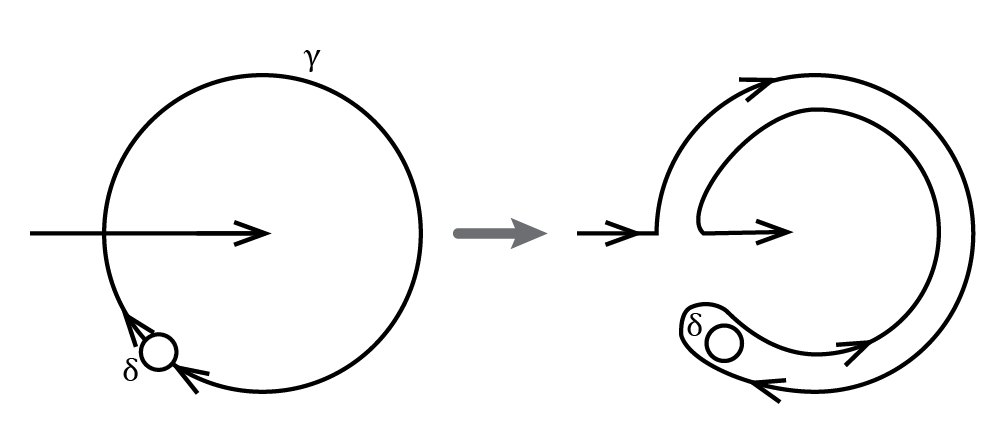}
	\end{center}
	\caption{}
\end{figure}

In the curve pushing case we have an added complication: the curve pushing map is defined by cutting $X$ along $\delta$ to form the surface $X_0$, pushing along the curve $\gamma$ in $X_0$, and regluing two copies of $\delta$ to re-form the surface $X$. As we see in figure $4$, calculating $\phi(c)$ is very similar to the the point pushing case, as long as $\alpha$ is a closed curve in $X_0$. We need to separately determine $\phi(\alpha)$ for arcs in $X_0$ whose endpoints lie on the two different copies of $\delta$. Let $\delta_1, \delta_2$ be those copies. We need to calculate the action of $\phi$ on $H_1(X, \delta_1 \cup \delta_2)$. 

Suppose first that $\gamma$ is a simple closed curve and $\alpha$ is an arc, one of whose endpoints lies on $\delta_1$. To obtain $\phi(\alpha)$ we add to $\alpha$ one strand parallel to $\gamma$ to $\alpha$ using the surgery depicted in figure $5$.  

\begin{figure}[h]
	\begin{center}
		\includegraphics[width=1\textwidth]{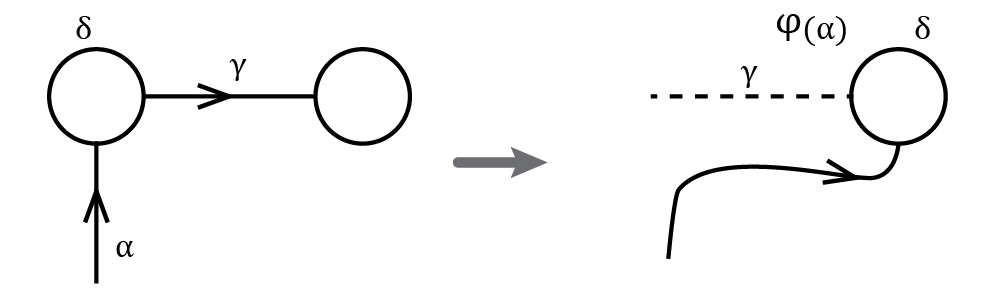}
	\end{center}
	\caption{}
\end{figure}

Now suppose that $\gamma$ has self intersections: $x_1, \ldots, x_s$ arranged from the beginning of $\gamma$ till its end. To find $\phi(\alpha)$ add strand parallel copy to $\gamma$ as above. Then, following the order $x_1, \ldots, x_s$, at $x_i$ add two parallel in opposite direction along the section of $\gamma$ exiting the intersection for the second time. Then perform the surgery described in Figure $6$. 

\begin{figure}[h]
	\begin{center}
		\includegraphics[width=1\textwidth]{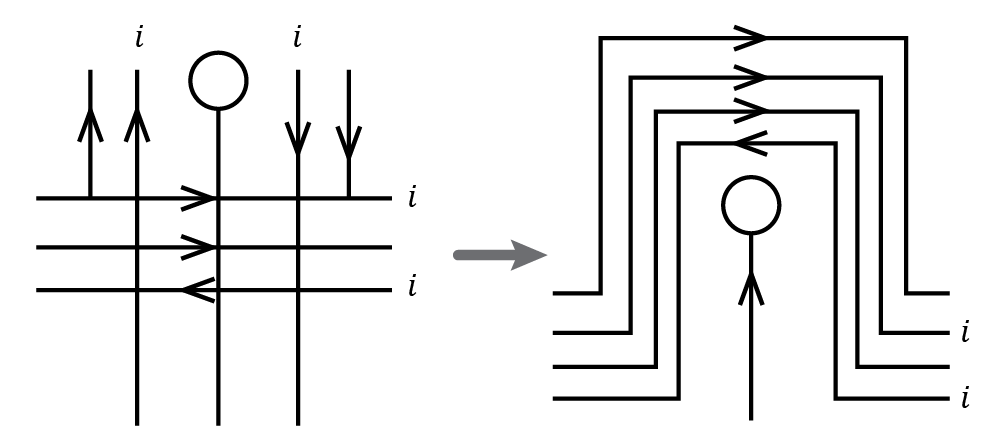}
	\end{center}
	\caption{}
\end{figure}

As opposed to the point pushing case, it is no longer the case that at each intersection there are an equal number of positively oriented and negatively oriented strands parallel to $\gamma$ in each direction. For each $x_j$, let $I_j = \wh{i}(e_j, f_j)$ where $e_j$ is the portion of $\gamma$ that hits the intersection first, and $f_j$ is the portion of $\gamma$ that hits the intersection second. Let $I_\gamma = \sum I_j$. This number determines how many copies of $d$ are added at the $j$-th intersection.  

Now let $\alpha$ be a closed curve in $X$. Putting the descriptions above together, we get that $$\phi(a) = a + \wh{i}(a, c) d + \wh{i}(a,d) \big( c + I_\gamma d \big) $$

Now let $\delta_1, \ldots, \delta_r$ be pairwise disjoint simple closed curves with homology classes $d_1, \ldots, d_r$ and take $\phi$ to be the multi-curve pushing map that pushes $\delta_j$ about $\gamma_j$ (whose homology class is denoted $c_j$). We get the following.

\begin{lemma} \label{multicurve}
	In the notation above, for any curve $\alpha$ in $X$: $$\phi(a) = a + \sum_j \wh{i}(a, c_j) d_j + \wh{i}(a, d_j) \big(c_j + I_{\gamma_j} d_j \big) $$
\end{lemma}

\subsection{Lifts of push maps to covers}

Point, multi-point, curve, and multi-curve pushing maps can have very complicated actions on the homology of covers of a surface.  For instant, the point pushing subgroup of $X$ contains many pseudo-Anosov elements (a point pushing map is pseudo-Anosov whenever the pushing curve fills the surface). By \cite{eigoffuc}, for any pseudo-Anosov element of the point pushing subgroup there is a finite cover $Y \to X$ to which $f$ lifts, such that the action of the lift of $f$ on $H_1(Y, \BZ)$ has eigenvalues off the unit circle. As a further example, by \cite{nonsolv} there is a finite cover $Y \to X$ where the image of the entire point pushing subgroup under the homological representation is non-solvable. 

This complexity is not immediately apparent from the descriptions given in Lemmas \ref{multipoint} and \ref{multicurve}. Part of the issue is that the lift of a point (resp. curve) pushing map to a cover need not be a point (resp. curve) or even a multi-point (resp. multi-curve) pushing map. In the point pushing case this is caused by the fact that if $Y \to X$ is a finite cover and $p$ is a puncture of $X$ then the covering map may be branched over $p$. If $p$ is pushed about the curve $\gamma$, and $\wt{p}$ is a lift of the puncture $p$ in $Y$, there may not be only one lift of the curve $\gamma$ originating at $\wt{p}$

Nevertheless, sometimes lifts of push maps to finite covers are themselves push maps, and their images under homological representations can thus be described by Lemmas \ref{multipoint} and \ref{multicurve}. 

We begin by describing a sufficient criterion for this to happen in the curve pushing case. Let $\delta$ be a simple closed curve in $X$. Identify $\delta$ with $\BR / \BZ$. Let $\gamma$ be a curve in $X$ originating at a point in $\delta$ corresponding to $p \in \BR/\BZ$ and terminating at the point $q = p + \frac{1}{2}$. 

Let $Y \to X$ be a regular finite cover. Let $\sigma$ be the element of the deck group of $Y\to X$ corresponding to $\delta$, and let $s$ be its order. Suppose that $\delta^s$ lifts to $Y$ for some $s$. Fix a lift $\wt{\delta}$ of $\delta^s$ and let $\wt{p_0}, \ldots, \wt{p_{s-1}}, \wt{q_0}, \ldots, \wt{q_{s-1}}  \in \BR/s\BZ$ be the lifts of $p$, $q$ in this $\wt{\delta}$ arranged cyclically. Lift the curve $\gamma$ to $\wt{\gamma}$ originating at $p_0$. Suppose this curve terminates at $\wt{q_j}$. 

We say that $\gamma$ satisfies the \emph{cyclic generation criterion} if the integer $j$ generates the group $\BZ/s\BZ$. If this is the case, then: $\bigcup\limits_{i=0}^{s-1}\sigma^i \wt{\gamma}$ is a $\langle \sigma \rangle$ invariant set consisting of a single curve which we call $\wt{\gamma}^s$.This curve passes through all the lifts $\wt{p_0}, \ldots, \wt{p_{s-1}}, \wt{q_0}, \ldots, \wt{q_{s-1}}$. The curve $\wt{\gamma}^s$ is simply a lift of $\gamma^s$ to $Y$, and by construction there is a unique such lift incident at the curve $\wt{\delta}$. Since $Y$ is regular, this is the case at every other lift of $\delta^s$.

Let $\phi \in \tup{Mod}(X)$ be the push map given by pushing $\delta$ about $\gamma^s$. By construction, the lift of $\phi$ to $Y$ is simply the multi-curve pushing map that pushes each lift of $\delta$ about each lift of $\wt{\gamma}^s$. Figure $7$ shows that if $\gamma$ does not satisfy the cyclic generation criterion there can be multiple lifts of the curve $\gamma^s$ at the loop $\wt{\delta}$ and hence the lift of $\phi$ is not a multi-curve pushing map. 

\begin{figure}[h]
	\begin{center}
		\includegraphics[width=1\textwidth]{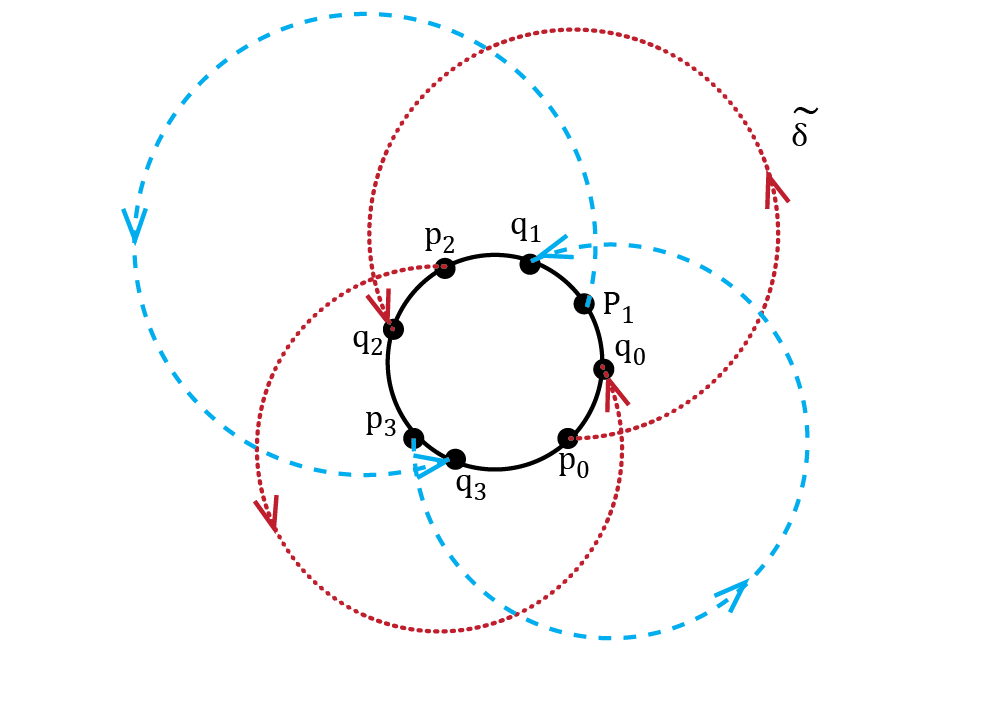}
	\end{center}
	\caption{}
\end{figure}

\section{proof of Theorem \ref{theorem2}} \label{proof}
We are now ready to construct push maps that mimic the properties of the map constructed in the proof of Theorem \ref{theorem3}. Let $X$ be a closed surface of genus $\geq 2$, and $X' = X \setminus \{\star\}$ for some basepoint $\star$. Let $\Gamma = \tup{Mod}(X')$. Let $Y \to X$ be a characteristic cover. Let $\rho$ be the corresponding homological representation.

\begin{lemma} \label{mainlemma}  There exist maps $\phi, \psi \in (\CI \setminus \CI_2) \cap \ker \rho$ such that $\phi$ is a point pushing map and $\psi$ is a curve pushing map whose image in $\tup{Mod}(X)$ is non-trivial. 
\end{lemma}

\begin{proof}

Pick a element of $\pi_1(X, \star)$ whose projection to $H = H_1(X, \BZ)$ is not zero. Let $\phi$ be the point pushing map about this curve. Since $\phi$ is a point pushing map, we have that $\phi \in \CI$. As stated in the introduction, we have that $\CI/\CI_2 \cong H \oplus \Lambda^3 /H$, where point pushing maps are sent to the first factor using the abelianization map $\pi_1(X, \star) \to H$. By our choice of a curve in $X$, we have that $\phi \notin \CI_2$. 

Finally, note that since the cover $Y \to X$ is not branched over the point $\star$, the lift of $\phi$ to $\tup{Mod}(Y)$ is a multi-point pushing map. Since $Y$ does not have punctures, Lemma \ref{multipoint} gives that $\phi \in \ker \rho$. 

To construct the map $\psi$, start with a separating simple closed curve $\delta$ in $X$ (we can choose one because the genus of $X$ is at least $2$). Let $X_0$ be the surface described in section \ref{pushmaps}, and let $Y_0 \to X_0$ be the corresponding (possibly branched) cover. Choose a curve $\gamma$ in $X_0$ with the following properties: 
\begin{enumerate}
	\item There exists a lift of $\gamma$ to $Y_0$ that is a closed curve that passes through every lift of $\delta$, and satisfies the cyclic generation criterion at each such lift. 
	\item The homology class $c$ of $\gamma$ is not equal to $0$. 
\end{enumerate}

Let $\psi \in \tup{Mod}(X)$ be the push map that pushes the curve $\delta$ along the curve $\gamma$. Let $T_\delta$ be the Dehn twist about $\delta$. By our assumption on $\gamma$, $\psi^s$ lifts to a multi-curve pushing map for some $s$. 

Let $\wt{\delta_1}, \ldots, \wt{\delta}_r$ be the lifts of $\delta$, and $\gamma_1, \ldots, \gamma_r$ be the curves about which they are being pushed. The curves $\gamma_1, \ldots, \gamma_r$ all have the same homology class - a lift of $\gamma^s$ which we denote by $\wt{\gamma}$. Since $\bigcup \limits_{j=1}^r \wt{\delta_j}$ separates $Y$, and $c_j$'s are all equal to each other, the term $\sum\limits_j  \wh{i}(a, d_j) \big( c_j + I_{\gamma_j} d_j \big)$ that appears in Lemma \ref{multicurve} is $0$. Similarly, since $\gamma$ passes through all the lifts of $\delta$ and these separate $Y$, we get that the term $\sum\limits_j \wh{i}(a,c_j) d_j$ 
 is also $0$. Lemma \ref{multicurve} now gives that $\psi \in \ker \rho$. This implies that $\psi \in \CI$.

It now remains to be seen that we can take $\psi \notin \CI_2$. Recall that $\delta$ separates $X$ into two subsurfaces - $X_0, X_1$, where $X_1$ is pushed about a curve in $X_0$. Let $a_1, b_1, \ldots, a_g, b_g$ be a standard generating set for $\pi_1(X, \star)$ such that $a_1, b_1, \ldots, a_j, b_j$ is a standard generating set for $\pi_1(X_1, \star)$. It is a standard calculation that the image of $\psi$ in $\CI/\CI_2 \cong \Lambda^3 H$ is $\sum \limits_{i=1}^j [a_1]\wedge [b_i] \wedge c$ (cf. Section 6.6.2 in \cite{FM}). Since this is not $0$, we have that $\psi \notin \CI_2$.

\end{proof}

\nid We are now ready to prove Theorem \ref{theorem2}. 

\begin{proof}  As we did before, decompose the $\tup{Sp}(2g,\BZ)$-representation  $\Lambda^3H$ into irreducible representations as $\Lambda^3 H \cong H \bigoplus \Lambda^3H/H $ where the first factor is the image of the point-pushing maps and the second factor is the image of maps that are not in the kernel of the forgetful map $\tup{Mod}(X') \to \tup{Mod}(X)$, where $X'$ is the surface obtained from $X$ by puncturing at $\star$. Lemma \ref{mainlemma} gives a point pushing map $\phi \in \CI$  that has non-zero image in the first factor and a curve pushing map $\psi \in \CI$ that has non-zero image in the second factor. 
Denote by $[\phi], [\psi]$ the images of $\phi$ and $\psi$ in $\CI/\CI_2$. Neither of these images is $0$. Thus, the $\tup{Sp}(2g,\BZ)$-orbits of $[\phi], [\psi]$ generate a finite index subgroup of $\CI/\CI_2$. This entire subgroup is in the kernel of the map $\CI/\CI_2 \to \rho(\CI)/\rho(\CI_2)$. This shows that $[\rho(\CI):\rho(\CI_2)] < \infty$. 

As in the proof of Theorem \ref{theorem3}, the $k \geq 3$ case follows from Lemma \ref{nilp}. 
\end{proof}

\section{Proof of Theorem \ref{theorem5}}
\begin{proof}
	Let $G$ be either the fundamental group of the closed surface $X$ with genus $\geq 2$ or $F_n$ with $n \geq 2$. Pick a prime $p$. Let $K$ be the kernel of the map $G \to H_1(G, \BZ/p\BZ)$. For any integer $m$, let $L_m \lhd K$ be the kernel of the map $K \to H_1(K, \BZ/p^m\BZ)$. Let $\rho$ be the homological representation corresponding to the cover given by the subgroup $K$. 
	
	Since $K$ is a characteristic subgroup of $G$ and $L_m$ is a characteristic subgroup of $K$, we get that $L_m \lhd K$. Furthermore, since $G/K$ and $K/L_m$ are $p$-group, we have that $G/L_m$ is a $p$-group, and is thus nilpotent. 
	
	Denote by $d(m)$ the nilpotence degree of $G/L_m$. By definition, the group $\CI_{d(m)}$ acts trivially on $G/L_m$. Thus, it implies that $\CI_{d(m)}$ acts trivially on $H_1(K, \BZ/p^m \BZ)$. This means that the elements of $\rho(\CI_{d(m)})$ are in the $p^m$-congruence subgroup of $\tup{GL}(H_1(K, \BZ))$. 
	
	Fix $k > 0$. There exists $g \in \CI_k$ and $i > 0$ such that $\rho(g)$ is not in the $p^i$-congruence subgroup of $\tup{GL}(H_1(K, \BZ))$. This implies that $\rho(\CI_k) \neq \rho(\CI_j)$ for every $j$ such that $j > d(i)$, $\rho(\CI_j) \neq \rho(\CI_k)$, as required.

	\end{proof}


\begin{thebibliography}{99}
\bibitem{FM} B. ~Farb and D. ~Margalit. \emph{A primer on mapping class groups, volume 49 of Princeton mathematical series.} Princeton University Press,  Princeton,  NJ, 2012.   ISBN 978-0-691-14794-9.
\bibitem{GL} F. ~Gr\"unewald and A. ~Lubtozky. \emph{Linear representations of the automorphism group of a free group.} Geom. Func. Anal. 18(5):1564--1608, 2009. 
\bibitem{GLLM} F. ~Gr\"unewald, M. ~Larsen, A. ~Lubotzky, and J. Malestein. \emph{Arithmetic quotients of the mapping class group.} To appear in Geom. Func. Anal.
\bibitem{eigoffuc} A. ~Hadari. \emph{Homological eigenvalues of lifts of pseudo-Anosov mapping classes to finite covers} (preprint) arxiv: 1712.01416
\bibitem{nonsolv} A. ~Hadari \emph{Non virtually solvable subgroups of mapping class groups have non virtually solvable representations}, to appear in Groups, Geom., and Dynamics. 

\bibitem{Johnson1} D. ~Johnson. \emph{An abelian quotient of the mapping class group $\CI_g$}. Math. Ann. 249(3):225--242, 1980. 
\bibitem{Johnson2} D. ~Johnson. \emph{Conjugacy relations in subgroups of the mapping class group and a group-theoretic description of the Rochlin invariant.} Math. Ann. 249(3): 243--263, 1980. 
\bibitem{Looij} E. ~Looijenga. \emph{Prym representations of mapping class groups.} Geom. Dedicata 64(1): 69--83, 1997. 

\bibitem{YiLiu} Yi. ~Liu. \emph{Virtual homological spectral radii for automorphisms of surfaces}, arxiv: https://arxiv.org/abs/1710.05039. 
\bibitem{McM} C. T ~McMullen. \emph{Braid groups and Hodge theory}. Math. Ann. 355(3): 893--946, 2012.



\end{thebibliography}
\end{document}